\documentclass{article}

\usepackage{amsmath,amssymb,amsfonts, amsthm}
\usepackage{ifthen}
\usepackage[capitalize]{cleveref}
\usepackage[T1]{fontenc}
\usepackage{bussproofs}
\usepackage{stmaryrd}
\usepackage{verbatim}
\usepackage{mathrsfs}
\usepackage{url}
\usepackage{authblk}

\newtheorem{theorem}{Theorem}[section]
\newtheorem{lemma}[theorem]{Lemma}
\newtheorem{proposition}[theorem]{Proposition}
\newtheorem{corollary}[theorem]{Corollary}

\theoremstyle{definition}

\theoremstyle{remark}
\newtheorem{remark}[theorem]{Remark}

\numberwithin{equation}{section}

\newcommand{\Id}{\mathrm{Id}}
\newcommand{\Var}{\mathrm{Var}}
\newcommand{\ARatMP}{\mathbf{ARatMP}}
\newcommand{\UT}{\mathrm{UT}}
\newcommand{\N}{\mathbb{N}}
\newcommand{\Z}{\mathbb{Z}}
\newcommand{\GKP}[2]{\mathbf{GKP}(#1)}
\newcommand{\KP}[2]{\mathbf{KP}(#1)}
\newcommand{\SSP}[2]{\mathbf{SSP}(#1)}

\newcommand{\autstep}[1][YYY]{\ifthenelse{\equal{#1}{YYY}}{\rightarrow}{\rightarrow_{#1}}}
\newcommand{\autsteps}[1][YYY]{\ifthenelse{\equal{#1}{YYY}}{\rightarrow^*}{\rightarrow^*_{#1}}}
\newcommand{\autstepsn}[2][YYY]{\ifthenelse{\equal{#1}{YYY}}{\rightarrow^{#2}}{\rightarrow^{#2}_{#1}}}
\newcommand{\grammarstep}[1][YYY]{\ifthenelse{\equal{#1}{YYY}}{\Rightarrow}{\Rightarrow_{#1}}}
\newcommand{\grammarsteps}[1][YYY]{\ifthenelse{\equal{#1}{YYY}}{\Rightarrow^*}{\Rightarrow^*_{#1}}}
\newcommand{\grammarstepsn}[2][YYY]{\ifthenelse{\equal{#1}{YYY}}{\Rightarrow^{#2}}{\Rightarrow^{#2}_{#1}}}

\begin{document}

\title{Knapsack and subset sum problems in nilpotent, polycyclic, and co-context-free groups}

\author[1]{Daniel K\"{o}nig\thanks{koenig@eti.uni-siegen.de}}
\author[1]{Markus Lohrey\thanks{lohrey@eti.uni-siegen.de}}
\author[2]{Georg Zetzsche\thanks{zetzsche@cs.uni-kl.de}}
\affil[1]{University of Siegen, Germany}
\affil[2]{Fachbereich Informatik\\Technische Universit\"{a}t Kaiserslautern}

\maketitle

\section{Introduction}

In their paper \cite{MyNiUs14}, Myasnikov, Nikolaev, and Ushakov started the investigation of classical 
discrete integer optimization problems in general non-commutative groups.
Among other problems, they introduced for a finitely generated (f.g.) group $G$ 
the {\em knapsack problem} and the {\em subset sum problem}.
The input for the knapsack problem is a sequence of group elements $g_1, \ldots, g_k, g \in G$ 
and it is asked whether there exists a solution
$(x_1, \ldots, x_k) \in \mathbb{N}^k$
of the equation $g_1^{x_1} \cdots g_k^{x_k} = g$. For the subset sum problem one restricts the solution
to $\{0,1\}^k$.
For the particular case $G = \mathbb{Z}$  (where the additive notation 
$x_1 \cdot g_1 + \cdots + x_k \cdot g_k = g$ is usually prefered)
these problems are {\sf NP}-complete if the numbers $g_1, \ldots, g_k,g$ are 
encoded in binary representation.
For subset sum this is shown in Karp's classical paper \cite{Karp72}. The statement for 
knapsack (in the above version) can be found in \cite{Haa11}.

In \cite{MyNiUs14} the authors enocde elements of the finitely generated group $G$ by words over the group generators
and their inverses. For $G = \mathbb{Z}$ this representation corresponds to the unary encoding of integers. It is known that
for unary encoded integers, knapsack and subset sum over $\Z$ can be both solved in polynomial time, and the precise 
complexity is $\mathsf{DLOGTIME}$-uniform $\mathsf{TC}^0$ \cite{ElberfeldJT11}, which is a very small complexity class
that roughly speaking captures the complexity of multiplying binary coded integers.
In \cite{MyNiUs14}, Myasnikov et al.~proved the following
new results:
\begin{itemize}
\item Subset sum and knapsack can be solved in polynomial time for every hyperbolic group. 
\item Subset sum for a virtually nilpotent group (a finite extension of a nilpotent group) can be solved in polynomial time. 
\item For the following groups, subset sum is {\sf NP}-complete (whereas the word problem can be solved in polynomial time):
free metabelian non-abelian groups of finite rank, the wreath product $\mathbb{Z} \wr \mathbb{Z}$, Thompson's group $F$,
and the Baumslag-Solitar group $\mathrm{BS}(1,2)$.
\end{itemize}
In this paper, we continue the investigation of knapsack and subset sum for arbitrary groups.
We prove the following results, where as in \cite{MyNiUs14} group elements are represented by finite words
over the group generators and their inverses:
\begin{itemize}
\item For every virtually nilpotent group, subset sum belongs to {\sf NL} (nondeterministic logspace).
\item There is a polycyclic group with an {\sf NP}-complete subset sum problem.
\item There is a nilpotent
group of class two for which knapsack is undecidable. This nilpotent group is 
a direct product of sufficiently many copies of the discrete Heisenberg group $H_3(\mathbb{Z})$.
In \cite{Lohrey15unitri}, the second author proved that there exists a nilpotent group (of large class) for 
which knapsack is undecidable. Here we improve this result to class two and at the same time simplify the construction
from \cite{Lohrey15unitri}. As a byproduct of our construction, we show that there exists a fixed nilpotent group of class two
together with four finitely generated abelian subgroups $G_1, G_2, G_3, G_4$ such that membership in the product $G_1G_2G_3G_4$
is undecidable. It is known that membership in a product of two subgroups of a polycyclic group is decidable \cite{LeWi79}.
\item The knapsack problem for the the discrete Heisenberg group $H_3(\mathbb{Z})$ is decidable. In particular, 
together with the previous point it follows that decidability
of knapsack is not preserved under direct products. 
\item The class of groups with a decidable knapsack problem is closed under finite extensions.
\item The knapsack problem is decidable for every co-context-free group. Recall that a group is co-context-free
if the complement of the word problem is a context-free language \cite{HoReRoTh2005}. 
\end{itemize}

\section{Nilpotent and polycyclic groups}

Let $A$ be a square matrix of dimension $d$ over some commutative ring  $R$. With $A[i,j]$ we denote the entry of $A$ in row $i$ and column $j$.
The matrix $A$ is called \emph{triangular} if $A[i,j] = 0$ whenever $i > j$, i.e., all entries below the main diagonal are $0$.
A \emph{unitriangular matrix} is a triangular matrix $A$ such that $A[i,i] = 1$ for all $1 \leq i \leq d$, i.e., all entries on the main 
diagonal are $1$. We denote the set of unitriangular matrices of dimension $d$ over the ring $R$ by $\UT_d(R)$. It is well known that for every 
commutative ring $R$, the set  $\UT_d(R)$ is a group (with respect to matrix multiplication).

An \emph{$n$-step solvable group} $G$ is a group $G$ that has a 
a subnormal series $G = G_n \rhd G_{n-1} \rhd G_{n-2} \rhd \cdots \rhd G_1 \rhd G_0 = 1$
(i.e., $G_i$ is a normal subgroup of $G_{i+1}$ for all $0 \leq i \leq n-1$) such that 
every quotient $G_{i+1}/G_i$ is abelian ($0 \leq i \leq n-1$).
If every quotient $G_{i+1}/G_i$ is cyclic, then $G$ is called {\em polycyclic}.
The number of $0 \leq i \leq n-1$ such that $G_{i+1}/G_i \cong \mathbb{Z}$ is called the \emph{Hirsch length} of $G$; it does 
not depend on the chosen subnormal series.
If $G_{i+1}/G_i \cong \mathbb{Z}$ for all $0 \leq i \leq n-1$ then $G$ is called \emph{strongly polycyclic}. 
The following characterizations of the class of polycyclic groups are known:
\begin{itemize}
\item A group is polycyclic if and only if it is solvable and every subgroup is finitely generated. 
\item A group is polycyclic if and only if it is a solvable group of integer matrices; this is a famous result by
Auslander and Swan \cite{Aus67,Swa67} .
In particular, every polycyclic group is linear, i.e., can be embedded into a matrix group over some field.
\end{itemize}
For a group $G$ its \emph{lower central series} is the series $G = G_0 \rhd G_1 \rhd G_2 \rhd \cdots$ of subgroups,
where $G_{i+1}  = [G_i,G]$, which is the subgroup generated by all commutators $[g,h]$ with $g \in G_i$ and $h \in G$.
Indeed, $G_{i+1}$ is a normal subgroup of $G_i$. The group $G$ is \emph{nilpotent of class $c$}, if
$G_c = 1$. Every f.g. nilpotent group is polycyclic, and every group $\UT_d(\mathbb{Z})$
$(d \geq 1)$ is f.g.  nilpotent of class $d-1$. 

The group $\UT_3(\mathbb{Z})$ is also denoted by $H_3(\mathbb{Z})$ 
and called the {\em discrete Heisenberg group}. 
Thus, $H_3(\mathbb{Z})$ is the group of all $(3 \times 3)$-matrices of the form
$$
\left(
\begin{array}{ccc}
1 & a & c \\
0 & 1 & b \\
0 & 0 & 1
\end{array}
\right)
$$
for $a,b,c, \in \mathbb{Z}$. The center $Z(H_3(R))$ of this group consists of all matrices of the form
$$
\left(
\begin{array}{ccc}
1 & 0 & c \\
0 & 1 & 0 \\
0 & 0 & 1
\end{array}
\right)
$$
for $c \in \mathbb{Z}$. The group $H_3(\mathbb{Z})$ is  nilpotent  of class two
(it is in fact  the free nilpotent group of class two and rank two). In other words, every
commutator $A B A^{-1} B^{-1}$ ($A, B \in H_3(\mathbb{Z})$) belongs to the center 
$Z(H_3(\mathbb{Z}))$. The  identity $(3 \times 3)$-matrix will be denoted
by $\Id_3$. 
Clearly, a direct product of copies of $H_3(\mathbb{Z})$  and $\mathbb{Z}$ is 
also nilpotent of class two.

We need the following results about nilpotent groups:


\begin{theorem}[Theorem 17.2.2 in \cite{KaMe79}] \label{thm-nilpotent-subgroup}
Every f.g. nilpotent group $G$ has a torsion-free normal subgroup $H$ of finite index (which is also f.g. nilpotent).
\end{theorem}

\begin{theorem}[Theorem 17.2.5 in \cite{KaMe79}] \label{thm-embed-nilpotent}
For every torsion-free f.g nilpotent group $G$ there exists $d \geq 1$ such that 
$G$ can be embedded into $\UT_d(\mathbb{Z})$.
\end{theorem}
A group is {\em virtually nilpotent} if it has a nilpotent subgroup of finite index.

\section{Subset sum and knapsack problems in groups}

Let $G$ be a f.g.~group, and fix an arbitrary finite generating set $\Sigma$ for $G$.
In this paper, we consider the following computational problems for $G$, where
elements of $G$ are represented by finite words over $\Sigma \cup \Sigma^{-1}$:
\begin{itemize}
\item Subset sum problem for $G$  (briefly $\SSP{G}{X}$): Given $g_1,\ldots, g_k, g\in G$, decide whether there exist $\varepsilon_1,\ldots,\varepsilon_k\in\{0,1\}$ such that $g = g_1^{\varepsilon_1}\cdots g_k^{\varepsilon_k}$.
\item Knapsack problem for $G$ (briefly $\KP{G}{X}$): 
Given $g_1,\ldots, g_k, g\in G$, decide whether there exist natural numbers $e_1,\ldots,e_k\geq 0$ such that $g =   g_1^{e_1}\cdots g_k^{e_k}$.
\end{itemize}
These problems were studied for general f.g. groups in \cite{MyNiUs14,FrenkelNU14}, where among others the following results
were shown:
\begin{itemize}
\item The subset sum problem for every f.g. virtually nilpotent group can be solved in polynomial time \cite{MyNiUs14}.
\item The subset sum problem and the knapsack problem for every hyperbolic group can be solved in polynomial time \cite{MyNiUs14}.
\item The knapsack problem can be solved in polynomial time in any free product of hyperbolic groups and finitely generated abelian groups 
\cite{FrenkelNU14}.
\item The subset sum problem for the following groups is {\sf NP}-complete: 
$\mathbb{Z} \wr \mathbb{Z}$, free metabelian (but non-abelian) groups of finite rank, and Thompson's group $F$ \cite{MyNiUs14}.
\end{itemize}
There is a variant of knapsack, where we ask wether for given 
$g_1,\ldots, g_k, g\in G$,  there exist {\em integers} $e_1,\ldots,e_k\in \mathbb{Z}$ such that $g =   g_1^{e_1}\cdots g_k^{e_k}$, i.e.,
whether $g$ belongs belongs to the product of cyclic groups $\langle g_1 \rangle \cdots \langle g_k \rangle$. This second version
is reducible to the above version with exponents from $\mathbb{N}$: Simply replace $g_i^{e_i}$ (with $e_i$ from $\mathbb{Z}$)
by $g_i^{c_i} (g_i^{-1})^{d_i}$ (with $c_i,d_i$ from $\mathbb{Z}$). We will prove undecidability results for the ``easier'' version with
integer quotients, whereas decidability results will be shown for the harder version with positive exponents.

\section{Subset sum problems in nilpotent groups}

In this section, we show that the subset sum problem for a finitely generated virtually nilpotent group
belongs to nondeterministic logspace ({\sf NL}). This is the class of all problems that can be solved
on a {\em nondeterministic} Turing-machine with a working tape of length $O(\log n)$, where $n$ is 
the length of the input, see e.g. \cite{AroBar09} for details.
Actually, we consider a  problem more general than the subset sum problem:
the membership problem for acyclic finite automaton, which was also studied in \cite{FrenkelNU14}.

Recall that a \emph{finite (nondeterministic) automaton} over a finite alphabet $\Sigma$ is a tuple 
$\mathcal{A} = (Q, \Delta, q_0, F)$, where 
\begin{itemize}
\item $Q$ is a finite set of states, 
\item $\Delta \subseteq Q \times \Sigma^* \times Q$
is a finite set of transitions, 
\item $q_0 \in Q$ is the initial state, and
\item $F \subseteq Q$ is the set of final states.
\end{itemize}
If the directed graph $(Q, \{ (p,q) \mid \exists w\in \Sigma^* : (p,w,q) \in \Delta \})$
has no directed cycle, then the  finite automaton $\mathcal{A}$ is \emph{acyclic}. 
An accepting run for a word $w$ is a sequence of transitions
$(q_0, w_1, q_1), (q_1, w_2, q_2), \ldots, (q_{n-1},w_n,q_n) \in \Delta$ such that 
$w = w_1 w_2 \cdots w_n$ and  $q_n \in F$. The language
$L(A) \subseteq \Sigma^*$ is the set of all words over $\Sigma$ that have an accepting run.
By splitting transitions, one can compute in logspace from a finite automaton $\mathcal{A}$ 
an automaton $\mathcal{B}$ such that $L(\mathcal{A}) = L(\mathcal{B})$ and all transitions of $\mathcal{B}$
are from $Q \times (\Sigma \cup \{\varepsilon\}) \times Q$. Moreover, $\mathcal{B}$ is acyclic if $\mathcal{A}$ is acyclic.

Let $G$ be a finitely generated group, and let $\Sigma$ be a finite group generating set for $G$.
Hence, $\Sigma \cup \Sigma^{-1}$ generates $G$ as a monoid and there is a canonical homomorphism
$h : (\Sigma \cup \Sigma^{-1})^* \to G$. For a finite automaton $\mathcal{A}$ over $\Sigma \cup \Sigma^{-1}$
and a word $x \in (\Sigma \cup \Sigma^{-1})^*$ we also write $x \in_G L(\mathcal{A})$ for
$h(x) \in h(L(\mathcal{A}))$.
The {\em acyclic rational subset membership problem for $G$} (briefly $\ARatMP(G)$) is the following computational problem:

\medskip
\noindent
Input: An acyclic finite automaton $\mathcal{A}$ over $\Sigma \cup \Sigma^{-1}$ and a word $x \in (\Sigma \cup \Sigma^{-1})^*$.\\
Question: Does $x \in_G L(\mathcal{A})$ hold?

\medskip
\noindent
Clearly, $\SSP{G}{X}$ is logspace reducible to $\ARatMP(G)$.

\begin{theorem} \label{thm-NL-UT}
For every $d \geq 1$, $\ARatMP(\mathsf{UT}_d(\mathbb{Z}))$ belongs to
{\sf NL}.
\end{theorem}

\begin{proof}
Let $\mathcal{A}$ be a finite automaton with $n$ states, whose transitions are labelled with generator matrices of 
$\mathsf{UT}_d(\mathbb{Z})$ or the identity matrix.
We nondeterministically guess a path of length at most $n$ from the initial state of  $\mathcal{A}$ to a final state of 
$\mathcal{A}$ and thereby multiply the matrices along the path. We only store the current state of 
$\mathcal{A}$, the product of the matrices seen so far, and the length of the path
travelled so far (so that after $n$ steps we can stop). The state of the automaton as well as the length of the path need
$O(\log n)$  bits. Hence,  we only have
to show that the product matrix can be stored in logarithmic space. For this, it suffices to show that
the matrix entries are bounded polynomially in $n$. Then, the binary coding of the matrix needs only $O(\log n)$
many bits (note that the matrix dimension $d$ is a constant). 
For this, we can use the following simple result (see \cite[Proposition~4.18]{Loh14} for a proof), which only holds
for unitriangular matrices: For a $(d \times d)$-matrix 
$M = (a_{i,j})_{1 \leq i,j \leq d}$ over $\mathbb{Z}$ let $|M| = \sum_{i=1}^d \sum_{j=1}^d |a_{i,j}|$.
Let $M_1, \ldots, M_n \in \mathsf{UT}_d(\mathbb{Z})$, $n \geq 2d$, and let
$m = \max \{ |M_i| \mid 1 \leq i \leq n \}$.
For the product of these matrices we have
$$
|M_1 M_2 \cdots M_n| \leq d + (d-1) \binom{n}{d-1} d^{2(d-2)} m^{d-1}.
$$
In our situation, the matrices $M_i$ are from a fixed set (generators and the identity matrix). Hence, $m$ and also $d$
are constants. Hence, the above bound is polynomial in $n$, which means that every entry of the product 
$M_1 M_2 \cdots M_n$ can be stored with $O(\log n)$  bits.
\end{proof}

\begin{theorem} \label{thm-ARatMP-finite-index}
Let $H$ be a finite index subgroup of the f.g. group $G$ (hence, $H$ is f.g. too).
Then $\ARatMP(G)$ is logspace-reducible to $\ARatMP(H)$.
\end{theorem}

\begin{proof}
Let $G$ and $H$ be as in the statement of the theorem. 
Let $\Gamma$ (resp., $\Sigma$) be a finite generating set for $G$ (resp., $H$). 
Let $Hg_0, Hg_1, \ldots, Hg_n$ be a list of all right cosets of $H$, where $g_0 = 1$.  
 
Let $\mathcal{A} = (Q, \Delta, q_0, F)$ be an acyclic finite automaton over the alphabet
$\Gamma \cup \Gamma^{-1}$ and let $x \in (\Gamma \cup \Gamma^{-1})^*$. We can assume that
$\Delta \subseteq Q \times (\Gamma \cup \Gamma^{-1}  \cup \{\varepsilon\}) \times Q$.
Assume that
$x = y g_s$ in $G$, where $y \in (\Sigma \cup \Sigma^{-1})^*$. We can compute the word $y$ 
and the coset representative $g_s$ in logspace as follows:
Let $x = a_1 a_2 \cdots a_m$. We store an index $i \in \{0,\ldots, n\}$,
which is initially set to $0$. Then, for $1 \leq j \leq m$ we do the following:
If $g_i a_j = w g_k$ for $w \in (\Sigma \cup \Sigma^{-1})^*$, then we append 
the word $w$ at the output tape and we set $i := k$. At the end, the word 
$y$ is written on the output tape and the final index $i$ is $s$ such that 
$x = y g_s$.

We now construct a new acyclic automaton $\mathcal{B}$ over the alphabet
$\Sigma \cup \Sigma^{-1}$ as follows:
\begin{itemize}
\item The state set is $Q \times \{ g_0, g_1 \ldots, g_n \}$.
\item Assume that $(p, a, q) \in \Delta$ is a transition of  $\mathcal{A}$ ($a \in  \Gamma \cup \Gamma^{-1} \cup \{\varepsilon\}$) and let $i \in \{0,1,\ldots,n\}$.
Assume that $g_i a = w g_j$ in $G$, where $w \in (\Sigma \cup \Sigma^{-1})^*$.
Then, we add the transition $( \langle p, g_i \rangle, w, \langle q, g_j \rangle)$ to $\mathcal{B}$.
\item The initial state of $\mathcal{B}$ is $\langle q_0, g_0 \rangle$.
\item The set of final states of $\mathcal{B}$ is  $F \times \{ g_s \}$.
\end{itemize}
From the construction, we get $x \in_G L(\mathcal{A})$ if and only if 
$y \in_H  L(\mathcal{B})$.
\end{proof}

\begin{theorem} \label{thm-subset-sum-nilpotent}
Let $G$ be finitely generated virtually nilpotent. Then,  the problem $\ARatMP(G)$ is
{\sf NL}-complete.
\end{theorem}

\begin{proof}
Hardness for {\sf NL} follows immediately from the {\sf NL}-hardness of the graph reachability problem 
for acyclic directed graphs. For the membership in {\sf NL} 
let $G$ be finitely generated virtually nilpotent.  By Theorem~\ref{thm-nilpotent-subgroup} and
\ref{thm-embed-nilpotent}, $G$ has a finite index subgroup $H$ such that $H$
is isomorphic to a subgroup of $\mathsf{UT}_d(\mathbb{Z})$. W.l.o.g we assume that $H$ is a subgroup of 
$\mathsf{UT}_d(\mathbb{Z})$. Membership in {\sf NL} follows from 
Theorem~\ref{thm-NL-UT} and Theorem~\ref{thm-ARatMP-finite-index}.
\end{proof}
By Theorem~\ref{thm-subset-sum-nilpotent}, the subset sum problem for a 
finitely generated virtually nilpotent belongs to {\sf NL}. It is open, whether this
upper bound can be further improved. In particular, it is open whether the subset
sum problem for the Heisenberg group $H_3(\Z)$ can be solved in deterministic
logspace. Recall from the introduction that subset sum for $\Z$ (and unary encoded numbers) belongs to 
$\mathsf{DLOGTIME}$-uniform $\mathsf{TC}^0$ which is a subclass of 
deterministic logspace. This result generalizes easily to any f.g.~abelian group.

\section{Subset sum in polycyclic groups}

We show in this section that there exists a polycyclic group with an $\mathsf{NP}$-complete
subset sum problem, which is in sharp contrast to nilpotent groups (assuming $\mathsf{NL} \neq \mathsf{NP}$).
Let us start with a specific example of a polycyclic group. 
Consider the two matrices  \label{g_a-and-h}
\begin{eqnarray*}
g_a =\left( \begin{array}{cc} a & 0 \\ 0 & 1 \end{array} \right) \text{ and }  h = \left( \begin{array}{cc} 1 & 1 \\ 0 & 1 \end{array} \right),
\end{eqnarray*}
where $a \in \mathbb{R}$, $a \geq 2$. Let $G_a = \langle g_a, h \rangle \leq \mathsf{GL}_2(\mathbb{R})$.
Let us remark that, for instance, the group $G_2$ 
is not polycyclic, see e.g. \cite[p.~56]{Wehr73}. On the other hand, we have:

\begin{proposition}[c.f.~\cite{KonigL15}] \label{prop-a-polycyclic-group}
The group $G_{1+\sqrt{2}}$ is polycyclic.
\end{proposition}

\begin{theorem}
$\SSP{G_{1+\sqrt{2}}}{X}$ is {\sf NP}-complete.
\end{theorem}

\begin{proof}
Let $\alpha = 1+\sqrt{2}$.
We follow the standard proof for the {\sf NP}-completeness of subset sum for binary encoded integers.
But we will work with real numbers of the form
$$
x = \sum_{i=0}^{n} x_i  \cdot \alpha^{3i},
$$
where the $x_i$ are natural numbers with $0 \leq x_i \leq 5$. The numbers $x_i$
are uniquely determined by $x$ in the following sense:

\medskip
\noindent
{\em Claim 1:} If 
\begin{equation} \label{eq-alpha}
\sum_{i=0}^{n} x_i  \cdot \alpha^{3i} =  \sum_{i=0}^{m} y_i  \cdot \alpha^{3i} 
\end{equation}
with $x_0, \ldots x_{n}, y_0, \ldots, y_{m}\in\{0,1,\ldots,5\}$ and $x_n \neq 0 \neq y_m$,
then $n=m$ and $x_i = y_i$ for all $0 \leq i \leq n$.

\medskip
\noindent
{\em Proof of Claim 1.} Assume that the conclusion of the claim fails.
Then, by canceling $\alpha$-powers with highest exponent, 
we obtain from \ref{eq-alpha} an identity of the form
$$
\sum_{i=0}^{n} x_i  \cdot \alpha^{3i} =  \sum_{i=0}^{m} y_i  \cdot \alpha^{3i} 
$$
where $n > m$, $x_0, \ldots x_{n}, y_0, \ldots, y_{m}\in\{0,1,\ldots,5\}$ and $x_n \neq 0$.
In order to lead this to a contradiction, it suffices to show
$$
\alpha^{3n} > \sum_{i=0}^{n-1} 5  \cdot \alpha^{3i}.
$$
Indeed, we have
$$
\sum_{i=0}^{n-1} 5  \cdot \alpha^{3i} < \sum_{i=0}^{n-1} (\alpha^{3i} + \alpha^{3i+1} + \alpha^{3i+2}) = \sum_{i=0}^{3n-1} \alpha^i = 
\frac{\alpha^{3n} - 1}{\alpha-1} < \alpha^{3n} .
$$
Let us now take a 3CNF-formula $C = \bigwedge_{i=1}^m C_i$, where $C_i = (z_{i,1} \vee z_{i,2} \vee z_{i,3})$.
Every $z_{i,j}$ is a literal, i.e., a boolean variable or a negated boolean variable. Let $x_1, \ldots, x_n$ be the boolean variables
appearing in $C$.

We now define numbers $u_1, \ldots, u_{2n+2m}$, and $t$ as follows, where
$1 \leq i \leq n$ and $1 \leq j \leq m$:
\begin{eqnarray*}
 u_{2i-1}                 & = &  \alpha^{3i-3}  +  \sum_{x_i \in C_k}  \alpha^{3n + 3k - 3} \\
 u_{2i} & = &  \alpha^{3i-3}  +  \sum_{\overline{x}_i \in C_k}  \alpha^{3n + 3k - 3} \\
 u_{2n+2j-1} & = & u_{2n+2j}                =  \alpha^{3n + 3j - 3} \\
 t                   & = & \sum_{i=1}^n \alpha^{3i-3}   +   \sum_{k=1}^m 3 \cdot \alpha^{3n+3k-3} 
\end{eqnarray*}
{\em Claim 2:}
$C$ is satisfiable if and only if there exists a subset $I \subseteq \{1, \ldots, 2n+2m\}$ such that
$\sum_{k \in I} u_k = t$.

\medskip
\noindent
{\em Proof of Claim 2.}  First assume that $C$ is satisfiable, and let 
$\varphi : \{x_1, \ldots, x_n \} \to \{0,1\}$ be a satisfying assignment for $C$.
We set $\varphi(\overline{x}_i) = 1 - \varphi(x_i)$.
For every clause $C_j = (z_{j,1} \vee z_{j,2} \vee z_{j,3})$ 
let $\gamma_j = |\{ k \in \{1,2,3\} \mid \varphi(z_{j,k})=1 \}|$ be the number of literals
in $C_j$ that are true under $\varphi$. Thus, we have $1 \leq \gamma_j \leq 3$.

We define the set $I$ as follows, where
$1 \leq i \leq n$ and $1 \leq j \leq m$:
\begin{itemize}
\item  $2i-1 \in I$ iff $\varphi(x_i) = 1$
\item $2i \in I$ iff $\varphi(x_i) = 0$
\item  If $\gamma_j = 3$, then $2n+2j-1 \not\in I$ and $2n+2j \not\in I$.
\item  If $\gamma_j = 2$, then $2n+2j-1 \in I$ and $2n+2j \not\in I$.
\item  If $\gamma_j = 1$, then $2n+2j-1  \in I$ and $2n+2j \in I$.
\end{itemize}
With this set $I$ we have indeed $\sum_{k \in I} u_k = t$.

For the other direction, let  $I \subseteq \{1, \ldots, 2n+2m\}$ such that
$\sum_{k \in I} u_k = t$. 
Note that in the sum $\sum_{k \in I} u_k$ no power $\alpha^{3k}$ can appear more than 5 times
(a power $\alpha^{3n+3j-3}$ with $1 \leq j \leq m$ can appear at most 5 times, since it appears in 3 of the numbers
$u_1, \ldots, u_{2n}$ and in 2 of the numbers $u_{2n+1}, \ldots, u_{2n+2m}$). This allows to use Claim 1.
 A comparision of $t$ and $\sum_{k \in I} u_k$ shows that 
 either $2i-1 \in I$ or $2i \in I$. We define the assignment $\varphi : \{x_1, \ldots, x_n \} \to \{0,1\}$ as follows:
\begin{itemize}
\item  $\varphi(x_i) = 1$ iff $2i-1 \in I$
\item $\varphi(x_i) = 0$ iff $2i \in I$ 
\end{itemize} 
As above, let  $\gamma_j$ be  the number of literals
in $C_j$ that are true under $\varphi$. 
Moreover, let $\delta_j = | I \cap \{ 2n+2j-1, 2n+2j\}|$ for $1 \leq j \leq m$.
We get
$$
\sum_{k \in I} u_k = \sum_{i=1}^n \alpha^{3i-3}   +    \sum_{j=1}^m (\gamma_j + \delta_j) \cdot \alpha^{3n+3j-3} = t = \sum_{i=1}^n \alpha^{3i-3}   +   \sum_{j=1}^m 3 \cdot \alpha^{3n+3j-3} .
$$
Since $\delta_j \in \{0,1,2\}$ we must have $\gamma_j \geq 1$ for all 
$1 \leq j \leq m$. This shows that $\varphi$ satisfies $C$.

We now map each of the numbers  $u_1, \ldots, u_{2n+2m},t$ to a word over the generators $g_\alpha, h$ (and their inverses) of 
the polycyclic group $G_\alpha$.
First, for $i \geq 0$ let  us define 
$$
w_i = g_\alpha^i h  g_\alpha^{-i}
$$
In the group $G_\alpha$ we have
$$
w_i = \left( \begin{array}{cc} 1 & \alpha^i \\ 0 & 1 \end{array} \right) 
$$
Finally, take a number 
$Y = \sum_{i=0}^n y_i \cdot \alpha^i$.
We define the word 
$$
w_Y = \prod_{i=0}^n w_i^{y_i} .
$$
In the group $G_\alpha$ we have
$$
w_Y = \left( \begin{array}{cc} 1 & Y \\ 0 & 1 \end{array} \right) .
$$
The words $w_{u_1}, \ldots, w_{u_{2n+2m}}, w_t$ can be computed in polynomial time (even in logspace)
from the 3CNF-formula $C$. Moreover, the construction implies that 
 $C$ is satisfiable iff there exists a subset $I \subseteq \{1, \ldots, 2n+2m\}$ such that
$\sum_{k \in I} u_k = t$ iff
there are $\varepsilon_1, \ldots, \varepsilon_{2n+2m} \in \{0,1\}$ such that 
$w_{u_1}^{\varepsilon_1} \cdots w_{u_{2n+2m}}^{\varepsilon_{2n+2m}} = w_t$ in the group $G_\alpha$.
\end{proof}

\section{Knapsack problems in nilpotent groups}

The goal of this section is to prove that the knapsack problem is undecidable for a 
direct product of sufficiently many copies of $H_3(\mathbb{Z})$, which is 
nilpotent of class two.

\subsection{Exponential expressions}

Let $\mathcal{X}$ be a countably infinite set of variables.
An \emph{exponential expression} $E$  over a group $G$  is a formal product of the form 
$$E=g_1^{x_1} g_2^{x_2} \cdots g_l^{x_l}$$ 
with $x_1,\ldots,x_l \in \mathcal{X}$ and $g_1,\ldots, g_l \in  G$. 
We do not assume that $x_i \neq x_j$ for $i \neq j$.
The group elements $g_1, \ldots, g_l$ will be also called the 
\emph{base elements} of $E$. The \emph{length} of $E$ is $l$.
Let $\Var(E) = \{ x_1, \ldots, x_n\}$ be the set of variables that appear in $E$.
For a finite set $X$ with $\Var(E) \subseteq X \subseteq \mathcal{X}$ and
$g \in G$, the set of \emph{$X$-solutions} of the equation $E=g$ is the set of mappings
$$
S_X(E=g) = \{ \nu : X \to \mathbb{Z} \mid g_1^{\nu(x_1)} g_2^{\nu(x_2)} \cdots g_l^{\nu(x_l)} = g \text{ in } G \}.
$$ 
Note that not every variable from $X$ has to appear as an exponent in $E$.
We moreover set $S(E=g) = S_{\Var(E)}(E=g)$.

For every $1 \leq i \leq n$
consider an exponential expression $E_i$ over a group $G_i$.
Then we can define the exponential expression 
$E = \prod_{i=1}^n E_i$ over the group $G = \prod_{i=1}^n G_i$.
It is defined by replacing in $E_i$ every occurrence of a base element $g \in G_i$ by the corresponding element
$$
(\underbrace{1, \ldots, 1}_{\text{$i-1$ many}}, g, \underbrace{1, \ldots, 1}_{\text{$n-i$ many}}) \in G
$$
and taking the concatenation of the resulting exponential expressions.
With this definition, the following lemma is obvious.

\begin{lemma} \label{lemma-direct-product}
For $1 \leq i \leq n$
let $E_i$ be an exponential expression over a group $G_i$.
Let $g_i \in G_i$ for $1 \leq i \leq n$. 
Let $X = \bigcup_{i=1}^n \Var(E_i)$.
Then for the exponential expression $E =  \prod_{i=1}^n E_i$ and the
element $g = (g_1, \ldots, g_n) \in \prod_{i=1}^n G_i$ we have:
$$
S_X(E=g) = \bigcap_{i=1}^n S_X(E_i = g_i).
$$
\end{lemma}

\begin{proposition} \label{prop-exponential-expression}
There are fixed constants $d,e\in \mathbb{N}$ and a fixed exponential expression $E$ over $G = H_3(\mathbb{Z})^d \times \mathbb{Z}^e$  such that the following problem is undecidable: 

\medskip
\noindent
Input: A element $g \in G$.\\
Question: Does $S(E=g) \neq \emptyset$ hold?
\end{proposition}

\begin{proof}
Let $P(x_1,\ldots,x_n) \in \mathbb{Z}[x_1,\ldots,x_n]$ be a fixed polynomial such that the following question is undecidable:

\medskip
\noindent
Input: A number $a \in \mathbb{N}$.\\
Question: Is there a tuple $(z_1,\ldots,z_n) \in \mathbb{Z}^n$ such that $P(z_1,\ldots,z_n)=a$.

\medskip
\noindent
By Matiyasevich's proof for the unsolvability of Hilbert's 10th problem,  we know that such a polynomial exists, see \cite{Mat93} for details.
By introducing additional variables, 
we can construct from the polynomial $P(x_1, \ldots, x_n)$  a system $\mathcal{S}$ of equations of the form
$x \cdot y = z$, $x + y = z$, $x = c$ (for $c \in \mathbb{Z}$) such that the equation $P(x_1, \ldots, x_n) = a$
has a solution in $\mathbb{Z}$ if and only if the system of equations $\mathcal{S}_a := \mathcal{S} \cup \{ x_0 = a \}$ has 
a solution in $\mathbb{Z}$. Here $x_0$ is a distinguished variable of $\mathcal{S}$.
Let $X$ be the set of variables that occur in $\mathcal{S}_a$.

Take an integer $a \in \mathbb{Z}$ (the input for our reduction).
Assume that $\mathcal{S}_a$ contains $d$ many equations of the form $x \cdot y = z$
and $e$ many equations of the form $x + y = z$ or $x = c$. Enumerate all equations
as $\mathcal{E}_1, \ldots, \mathcal{E}_{d+e}$, where
$\mathcal{E}_1, \ldots, \mathcal{E}_d$ are all equations of the form $x \cdot y = z$.
Let $G_i =  H_3(\mathbb{Z})$ for $1 \leq i \leq d$ and 
$G_i = \mathbb{Z}$ for $d+1 \leq i \leq d+e$
We define for every $1 \leq i \leq d+e$ an element $g_i \in G_i$ and an exponential expression $E_i$ over
$G_i$ as follows:

\medskip
\noindent
{\em Case 1.} $\mathcal{E}_i = (x \cdot y = z)$ and thus $G_i =  H_3(\mathbb{Z})$.
Then, we set $g_i = \Id_3$ (the identity matrix) and
\begin{align*}
E_ i =  
\left( \! \!
\begin{array}{ccc}
1 & 0 & 0 \\
0 & 1 & 1 \\
0 & 0 & 1
\end{array} \! \!
\right)^{\! \!  \!   x}
\left(  \! \!
\begin{array}{ccc}
1 & 1 & 0 \\
0 & 1 & 0 \\
0 & 0 & 1
\end{array} \! \!
\right)^{\! \!   \!  y} 
 \left( \! \!
\begin{array}{rrr}
1 & 0 & 0 \\
0 & 1 & -1 \\
0 & 0 & 1
\end{array} \! \! 
\right)^{\! \!  \!  x}
\left( \! \!
\begin{array}{rrr}
1 & -1 & 0 \\
0 & 1 & 0 \\
0 & 0 & 1
\end{array} \! \!
\right)^{\! \!  \!  y}
\left( \! \!
\begin{array}{rrr}
1 & 0 & 1 \\
0 & 1 & 0 \\
0 & 0 & 1
\end{array} \! \!
\right)^{\! \!  \!  z} .
\end{align*}
One can easily check that a mapping $\nu : X \to \mathbb{Z}$ is a solution of $E_i = g_i$ 
if and only if $\nu(x) \cdot \nu(y) = \nu(z)$.

\medskip
\noindent
{\em Case 2.} $\mathcal{E}_i = (x + y = z)$ and thus $G_i =  \mathbb{Z}$. Then,
$g_i = 0$ and 
$E_i$ is (written in additive form)
$E_i = x + y - z$ (or, written multiplicatively, $E_i = a^x a^y a^{-z}$, where $a$ is a generator of $\mathbb{Z}$).
Then, a mapping $\nu : X \to \mathbb{Z}$ is a solution of $E_i = g_i$ 
if and only if $\nu(x) + \nu(y) = \nu(z)$.

\medskip
\noindent
{\em Case 3.} $\mathcal{E}_i = (x = c)$ (this includes the distinguished equation $x_0 = a$)
and thus $G_i =  \mathbb{Z}$. Then, $g_i = c$ and
$E_i =  x$ (or, written multiplicatively, $E_i = a^x$).
Then, a mapping $\nu : X \to \mathbb{Z}$ is a solution of $E_i = g_i$ 
if and only if $\nu(x) = c$.

\medskip
\noindent
Let $E = \prod_{i=1}^d E_i$ and $g = (g_1, \ldots, g_d)$.
By Lemma~\ref{lemma-direct-product},  a mapping $\nu : X \to \mathbb{Z}$ is a solution of 
$E=g$ if and only if $\nu$ is a solution of the system $\mathcal{S}_a$.
Also note that $g \in G$ depends on the input integer $a$, 
but the exponential expression $E$ only depends on the fixed polynomial
$P(x_1, \ldots, x_n)$.
\end{proof}

\begin{remark} \label{remark-commutation}
The fixed exponential expression $E$ from Proposition~\ref{prop-exponential-expression} has the following property
that will be exploited in the next section: We can write $E = E_1 E_2 \cdots E_m$ such that every $E_i$ has length at most $4$
and every base element $g$ from $E_i$ commutes with every base element $h$ from $E_j$ whenever $i \neq j$.
For this, note that the last matrix in the exponential expression from Case 1 is central in $H_3(\mathbb{Z})$.
\end{remark}

\subsection{Undecidability of knapsack for nilpotent groups of class two}

Let $E = g_1^{x_1} g_2^{x_2} \cdots g_l^{x_l}$ be an exponential expression over the f.g. group $G$
and let $X = \Var(E)$.
Consider the group $G \times \mathbb{Z}^l$. 
For $1 \leq i \leq l$ let $e_i \in \mathbb{Z}^l$ be the $i$-th unit vector from $ \mathbb{Z}^l$.
For every $x \in X$ define 
$$
e_x = \sum_{1 \leq i \leq l, x_i = x} \!\!\! e_i \ \in \ \mathbb{Z}^l \quad\text{ and }\quad
h_x = (1, e_x) \in G \times \mathbb{Z}^l.
$$
Note that $h_x$ is central in $G \times \mathbb{Z}^l$.
Moreover, for $1 \leq i \leq l$ let
$$
h_i = (g_i, -e_i) \in G \times \mathbb{Z}^l.
$$
Then, for a given group element $g \in G$,
we have $S(E=g) \neq \emptyset$ if and only if 
$$
(g,0) \in \prod_{x \in X} \langle h_x \rangle \prod_{i=1}^l \langle h_i \rangle.
$$
By applying the above construction to the fixed exponential expression $E$ over the fixed group $G=H_3(\mathbb{Z})^d \times \mathbb{Z}^e$
from Proposition~\ref{prop-exponential-expression}, we obtain (note that $\mathbb{Z} \leq H_3(\mathbb{Z})$):

\begin{theorem} \label{thm-product-cyclic-undec}
There exist a fixed constant $d$ and a fixed list $g_1, \ldots, g_\lambda \in H_3(\mathbb{Z})^d$ of group elements such that
membership in the product  $\prod_{i=1}^{\lambda} \langle g_i \rangle$ is undecidable.
\end{theorem}
In particular, we have:

\begin{theorem} \label{thm-knapsack-undec}
There exists a  fixed constant $d$ such that $\KP{H_3(\mathbb{Z})^d}{X}$
is undecidable.
\end{theorem}
Finally, from the construction in the previous section, we also obtain the following 
undecidability result.

\begin{theorem} \label{thm-4-subgroups}
There exist a fixed constant $d$ and a fixed list of four abelian subgroups $G_1, G_2, G_3, G_4 \leq  H_3(\mathbb{Z})^d$  such that
membership in the product  $G_1 G_2 G_3 G_4$ is undecidable.
\end{theorem}

\begin{proof}
Recall from Remark~\ref{remark-commutation} that the 
exponential expression from Proposition~\ref{prop-exponential-expression} can be written as $E_1 E_2 \cdots E_m$ such that every $E_i$ has length at most $4$,
and every base element $g$ from $E_i$ commutes with every base element $h$ from $E_j$ whenever $i \neq j$. The above construction implies that
the sequence of group elements $g_1, g_2, \ldots, g_\lambda$ from Theorem~\ref{thm-product-cyclic-undec}
can be split into blocks $B_1, B_2, \ldots, B_\mu$ of length at most $4$ such that 
every group element from block $B_i$ commutes with every group element from $B_j$ whenever $i \neq j$.
This allows to rearrange the product of cyclic groups $\prod_{i=1}^{\lambda} \langle g_i \rangle$ as a product of four abelian subgroups
$G_1, G_2, G_3, G_4$, where $G_i$ is  generated by all group elements, which are at the $i$-th position in their block.
\end{proof}

\begin{remark}
In contrast to Theorem~ \ref{thm-4-subgroups}, it was shown in \cite{LeWi79} that a product of two subgroups of a polycyclic group
is closed in the profinite topology. Since polycyclic groups are finitely presented it follows that membership in a product of two
subgroups of a polycyclic group is decidable. This leaves open whether membership in a product of three subgroups of a 
polycyclic (or nilpotent) group is decidable.
\end{remark}

Let us finally prove that the knapsack problem for the discrete Heisenberg group $H_3(\mathbb{Z})$ is decidable.

\begin{theorem} \label{thm-Heisenberg-decidable}
For every $e \geq 0$,  $\KP{H_3(\mathbb{Z}) \times \mathbb{Z}^e}{X}$ is decidable.
\end{theorem}

\begin{proof}
Let us first show the result for $H_3(\mathbb{Z})$.
Take matrixes $A, A_1, \ldots, A_l \in H_3(\mathbb{Z})$ and let
$$
A = \left( 
\begin{array}{ccc}
1 & a & c \\
0 & 1 & b \\
0 & 0 & 1
\end{array} 
\right)
\qquad \text{and} \qquad
A_i = \left( 
\begin{array}{ccc}
1 & a_i & c_i \\
0 & 1 & b_i \\
0 & 0 & 1
\end{array} 
\right)
$$
A straightforward induction over $n$ shows that
$$
A_i^n =  \left( 
\begin{array}{ccc}
1 & a_i \cdot n & \ c_i \cdot n + a_i b_i \frac{(n-1) n}{2} \\
0 & 1 & b_i \cdot n \\
0 & 0 & 1
\end{array} 
\right)
$$
Hence, there is a solution $(x_1, \ldots, x_l) \in \mathbb{N}^l$ of 
$A = A_1^{x_1} \cdots   A_l^{x_l}$ if and only if the following system of three
Diophantine equations has a solution over $\N$:
\begin{gather*}
a = \sum_{i=1}^l a_i \cdot x_i \\
b = \sum_{i=1}^l b_i \cdot x_i \\
c = \sum_{i=1}^l c_i \cdot x_i  +  \sum_{i=1}^l  a_i b_i \frac{(x_i-1) x_i}{2} +  \sum_{1 \leq i < j \leq l}  a_i b_j x_i x_j
\end{gather*}
This is a Diophantine system with a single quadratic equation and two linear equations. 
By \cite{DuLiSha14}, a system consisting of a single quadratic Diophantine equation together with an arbitrary number of linear equations 
can be reduced to a single quadratic Diophantine equation, which has the same solutions over $\Z$. 
By \cite{GrSe04}, one can decide whether this quadratic Diophantine equation has a solution over $\N$.

Finally, the above proof also works for the group $H_3(\mathbb{Z}) \times \mathbb{Z}^e$, since 
we only get additional linear equations. 
\end{proof}

\begin{corollary} \label{coro-knapsack-direct-product}
The class of f.g. groups with a decidable knapsack problem is not closed under direct products.
\end{corollary}

\begin{proof}
This follows directly from Theorem~\ref{thm-knapsack-undec} and \ref{thm-Heisenberg-decidable}.
\end{proof}

\section{Knapsack problems for finite extensions}

We show that in contrast to direct products, decidability of the knapsack problem is preserved under
finite extensions.
For this, it will be convenient to consider a slightly extended version of the
knapsack problem, which we will prove equivalent (with respect to
polynomial time reducibility) to the knapsack problem. The \emph{generalized
knapsack problem} (briefly $\GKP{G}{X}$) is the following decision
problem: Given $g_1,\ldots,g_k\in G$ and $f_0,\ldots,f_k\in G$, decide
whether
\begin{equation} f_0g_1^{n_1}f_1g_2^{n_2}f_2\cdots g_k^{n_k}f_k=1 \label{gkpeq} \end{equation}
for some $n_1,\ldots,n_k\in\N$. An \emph{instance} of the generalized
knapsack problem is therefore a tuple $(f_0, g_1, f_1, \ldots ,
g_k, f_k)$ with $f_0,\ldots,f_k\in G$ and $g_1,\ldots,g_k\in G$.
If \eqref{gkpeq} holds, we call the tuple $(n_1,\ldots,n_k)$ a
\emph{solution}. If two instances have the same set of solutions, we
call them \emph{equivalent}.


\begin{proposition}\label{kpvsgkp}
$\KP{G}{X}$ and $\GKP{G}{X}$ are inter-reducible in polynomial time.
\end{proposition}
\begin{proof}
Since $g_1^{n_1}\cdots g_k^{n_k}=g$ if and only if
$g^{-1}g_1^{n_1}\cdots g_k^{n_k}=1$, $\KP{G}{X}$ clearly reduces to
$\GKP{G}{X}$ in polynomial time.

Let us reduce $\GKP{G}{X}$ to $\KP{G}{X}$. Let $(f_0,g_1,f_1,\ldots,
g_k,f_k)$ be an instance of $\GKP{G}{X}$. Observe that since
$g_i^{n_i}f_i=f_i(f_i^{-1}g_if_i)^{n_i}$, if we replace $f_{i-1}$,
$g_i$, and $f_i$ by $f_{i-1}f_i$, $f_i^{-1}g_if_i$ and $1$,
respectively, we obtain an equivalent instance in which $f_i=1$. By
repeating this step $k$ times, starting with $f_k$, we arrive at an
instance with $f_1=\cdots=f_k=1$. Then, clearly, $f_0g_1^{n_1}\cdots
g_k^{n_k}=1$ is equivalent to $g_1^{n_1}\cdots g_k^{n_k}=f_0^{-1}$.
\end{proof}

From now on, let $G$ be finitely generated and $H$ be a finite index subgroup
of $G$, which is therefore finitely generated too.  Furthermore, let
$R\subseteq G$ be a finite set of representatives of right cosets of $H$ in
$G$.  Then for each $g\in G$, there is a unique $\rho(g)\in R$ such that $g\in
H\rho(g)$.  Also recall from the proof of Theorem~\ref{thm-ARatMP-finite-index}
that from a given element $g \in G$ we can compute effectively a decomposition
$g = h r$ with $h \in H$ and $r \in R$.  This fact will be implicitly used
throughout this section.


\begin{lemma}\label{moveright}
Let $g_1,g_2\in G$ and $\rho(g_1g_2)=\rho(g_1)$. We can compute $h_1,h_2\in H$
and $r\in R$ such that $g_1g_2^t=h_1h_2^tr$ for every $t\ge 0$.
\end{lemma}
\begin{proof}
Since $\rho(g_1g_2)=\rho(g_1)$, we can write $g_1=h_1r$ and $g_1g_2=h_{12}r$ for $h_1, h_{12} \in H$ and $r \in R$.
Moreover, we can find $h_2\in H$ and $r_2\in R$ with $rg_2=h_2r_2$. Then
\[ h_{12}r=g_1g_2=h_1rg_2=h_1h_2r_2 \]
and hence $r_2=r$. This means $rg_2=h_2r$ and thus $rg_2^t=h_2^tr$ and
\[ g_1g_2^t=h_1rg_2^t=h_1h_2^tr. \]
\end{proof}

\begin{theorem}\label{finiteextensions}
Let $H$ be a finite-index subgroup of a finitely generated group $G$.
Then $\KP{G}{X}$ is decidable if and only if $\KP{H}{X}$ is decidable.
\end{theorem}
\begin{proof}
Since the ``only if'' direction is trivial, it remains to prove the
``if'' direction. According to \cref{kpvsgkp}, it suffices to show that
if $\GKP{H}{X}$ is decidable, then $\GKP{G}{X}$ is decidable.

We say that an instance $I=(f_0,g_1,f_1,\ldots,g_k,f_k)$ of $\GKP{G}{X}$
is \emph{$j$-pure} if $f_0,g_1,\ldots,f_{j-1},g_j\in H$. In particular,
every instance is $0$-pure. We call an instance \emph{pure} if it is
$k$-pure. If an instance is $j$-pure, but not $(j+1)$-pure, then $k-j$
is its \emph{impurity}.

First, we prove the following claim by induction on the impurity of $I$: For
every instance $I=(f_0,g_1,f_1,\ldots,g_k,f_k)$ of $\GKP{G}{X}$, we can
construct finitely many pure instances of $\GKP{G}{Y}$ such that the solution
set of $I$ is the union of affine images of their solution sets.

Suppose $I$ is $j$-pure but not $(j+1)$-pure.  Write $f_j=hr$ for $h\in H$ and
$r\in R$.  Since $R$ is finite, there are $m,\ell\in\N$ with
$\rho(rg_{j+1}^m)=\rho(rg_{j+1}^{m+\ell})$. We use \cref{moveright} to find
$h_1,h_2\in H$ and $r'\in R$ such that $rg_{j+1}^{m+t\ell}=h_1h_2^tr'$ for all
$t\ge 0$.  In particular
\[ f_jg_{j+1}^{m+t\ell}=hrg_{j+1}^{m+t\ell}=hh_1h_2^tr'. \]
We can also find for each $0\le s<m$ elements $\hat{h}_s\in H$ and
$\hat{r}_s\in R$ with $rg_{j+1}^s=\hat{h}_s\hat{r}_s$.  Finally, we can find
for each $0\le s<\ell$ a decomposition $r'g_{j+1}^s=\bar{h}_s\bar{r}_s$ with
$\bar{h}_s\in H$, $\bar{r}_s\in R$.  Note that each element $f_jg_{j+1}^n$ can
be written in one of the following forms:
\begin{align*}
f_jg_{j+1}^n&=h\hat{h}_s\hat{r}_s&&\text{for some $0\le s<m$,} \\
f_jg_{j+1}^n&=hh_1h_2^t\bar{h}_s\bar{r}_s&&\text{for some $0\le s<\ell$ and $t\ge 0$}.
\end{align*}
Here, the first equality holds if $n<m$ and the second one holds if $n\ge m$
and $n=m+t\ell+s$ with $0\le s<\ell$.

We therefore construct two types of instances.  The first type consists of the
instances
\[ (f_0,g_1,f_1,\ldots,g_{j-1},f_{j-1},g_j,h\hat{h}_s\hat{r}_sf_{j+1},g_{j+2},f_{j+2},\ldots,g_k,f_k), \]
for $0\le s<m$. The second type consists of instances
\[ (f_0, g_1,f_1, \ldots, g_{j-1},f_{j-1}, g_j,hh_1, h_2,\bar{h}_s\bar{r}_sf_{j+1}, g_{j+2},f_{j+2},\ldots, g_k,f_k) \]
for each $0\le s<\ell$. Observe that $I$ has a solution if and only if
one of these new instances has one. Furthermore, each of these instances
has lower impurity than $I$. Hence, the induction hypothesis yields the
desired finite set of instances. This proves our claim.

Let us now prove the \lcnamecref{finiteextensions}. Given an
instance $I$ of $\GKP{G}{X}$, we construct pure instances
$I_1,\ldots,I_m$ of $\GKP{G}{X}$ such that $I$ has a solution if and
only if one of $I_1,\ldots,I_m$ has one. Since $I_i$ is pure, if
$I_i=(f_0,g_1,f_1,\ldots,g_k,f_k)$, then $f_0,g_1,\ldots,f_{k-1},g_k\in
H$, but $f_k$ may not be in $H$. However, the equation
\[ f_0g_1^{n_1}f_1\cdots g_k^{n_k}f_k=1 \]
can only have a solution if $f_k\in H$. Moreover, if $f_k\in H$,
then $I$ is in fact an instance of $\GKP{H}{Y}$. Since we can decide
whether $f_k\in H$, we can pick from $I_1,\ldots,I_m$ those that are
instances of $\GKP{H}{Y}$. This means, from $I$ we have constructed
finitely many instances of $\GKP{H}{Y}$ such that $I$ has a solution
if and only if one of the new instances has one. This proves the
\lcnamecref{finiteextensions}.
\end{proof}

\section{Knapsack problems for co-context-free groups}\label{cocf}

In this \lcnamecref{cocf}, we exhibit another class of groups with a decidable
knapsack problem, namely co-context-free groups, which we introduce first. 

A \emph{language} is a subset of a free monoid $X^*$, where
$X$ is an \emph{alphabet}, i.e. a finite set of abstract symbols. A
\emph{context-free grammar} is a tuple $\Gamma=(N,T,P,S)$, where
\begin{itemize}
\item $N$ and $T$ are disjoint alphabets, their members are called \emph{nonterminals} and \emph{terminals}, respectively,
\item $P\subseteq N\times (N\cup T)^*$ is a finite set of \emph{productions},
\item $S\in N$ is the \emph{start symbol}.
\end{itemize}
A production $(A,w)\in P$ is also denoted $A\to w$.  In a context-free grammar,
the productions allow us to rewrite words.  Specifically, for $u,v\in (N\cup
T)^*$, we write $u\grammarstep[\Gamma] v$ if there are $x,y\in (N\cup T)^*$ such
that $u=xAy$ and $v=xwy$ for some production $A\to w$ in $P$. Furthermore,
$\grammarsteps[\Gamma]$ denotes the reflexive transitive closure of
$\grammarstep[\Gamma]$. The language \emph{generated by $\Gamma$} is then defined as
\[ L(\Gamma)=\{w\in T^* \mid S\grammarsteps[\Gamma] w \}. \]
A language is called \emph{context-free} if it is generated by some
context-free grammar.  

Let $\Sigma$ be a finite generating set of the group $G$ and let
$h\colon(\Sigma\cup \Sigma^{-1})^*\to G$ be the canonical monoid homomorphism.
The \emph{word problem} and the \emph{co-word problem} \emph{(with respect to
$\Sigma\cup\Sigma^{-1}$)} of $G$ are the languages
\begin{align*}
&\{ w\in (\Sigma\cup\Sigma^{-1})^* \mid h(w)=1 \} ~~~\text{and} \\
&\{ w\in (\Sigma\cup\Sigma^{-1})^* \mid h(w)\ne 1 \}
\end{align*}
respectively. Since it does not depend on the chosen generating set whether the
word problem or the co-word problem are context-free~\cite{HoReRoTh2005}, we
may define a group $G$ to be \emph{(co-)context-free} if its (co-)word problem
is a context-free language. Co-context-free groups were introduced by Holt,
Rees, R\"{o}ver, and Thomas~\cite{HoReRoTh2005} and shown to significantly
extend the class of context-free groups (which are, by a well-known result of
Muller and Schupp and Dunwoody, precisely the virtually free
groups~\cite{MullerSchupp1983,Dunwoody1985}): The class of co-context-free
groups is closed under taking direct products, taking restricted standard
wreath products with a context-free top-group, passing to finitely generated
subgroups and finite index overgroups.  Furthermore, Lehnert and
Schweitzer~\cite{LeSch2007} have shown that the Higman-Thompson groups are
co-context-free as well.

\begin{theorem}\label{kpcocf}
Every co-context-free group has a decidable knapsack problem.
\end{theorem}

Note that this means in particular that the wreath product $\Z\wr\Z$ has
a decidable knapsack problem, which is in contrast to the fact that this
group has an undecidable submonoid membership problem~\cite{LoStZe2015}.
\begin{proof}[Proof of \cref{kpcocf}]
Let $W$ be the co-word problem of $G$ with respect to $\Sigma\cup\Sigma^{-1}$
and let $W$ be context-free. 

We will need some terminology. A language $L\subseteq X^*$ is called
\emph{regular} if it can be obtained from the empty set and the
singletons $\{x\}$, $x\in X$, by the operations 
\begin{itemize}
\item \emph{union}, which turns $K\subseteq X^*$
and $M\subseteq X^*$ into $K\cup M$,
\item \emph{concatenation}, which turns
$K,M\subseteq X^*$ into $\{uv \mid u\in K,~v\in M\}$, and
\item \emph{iteration}, which maps $M\subseteq X^*$ to the submonoid of $X^*$
generated by $M$. 
\end{itemize}
For every context-free language $L\subseteq X^*$,
homomorphisms $\alpha\colon X^*\to Y^*$ and $\beta\colon Z^*\to X^*$ and regular
language $K\subseteq X^*$, the languages
$\alpha(L)$,  $\beta^{-1}(L)$, and $L\cap K$
are context-free as well and we can effectively compute a grammar for the
resulting languages~\cite{Berstel1979}. 

Suppose we are given $g_1,\ldots,g_k,g$ as an instance of the knapsack problem.
and let these elements be written as words $w_1,\ldots,w_k,w$, respectively,
over $\Sigma\cup\Sigma^{-1}$. Consider the alphabets $X=\{a_1,\ldots,a_k\}$,
$Y=X\cup \{a\}$, and the homomorphisms $\alpha\colon Y^*\to
(\Sigma\cup\Sigma^{-1})^*$, with $\alpha(a_i)=w_i$ for $1\le i\le k$ and
$\alpha(a)=w^{-1}$. Here, $w^{-1}$ is the word obtained by inverting the
generators and then reversing the word. Furthermore, observe that the language
$K=\{a_1\}^*\cdots\{a_k\}^*\{a\}$ is regular. Moreover, let $\beta\colon Y^*\to
X^*$ be the homomorphism with $\beta(a_i)=a_i$ for $1\le i\le k$ and
$\beta(a)=\varepsilon$. Then, the language
\[ M=\beta(\alpha^{-1}(W)\cap K)=\{ a_1^{e_1}\cdots a_k^{e_k} \mid g_1^{e_1}\cdots g_k^{e_k}\ne g\} \]
is effectively context-free. Clearly, there exist $e_1, \ldots, e_k \in \N$ with 
$g_1^{e_1} \cdots g_k^{e_k} = g$ if and only if  $M \neq \{a_1\}^*\cdots \{a_k\}^*$.
In order to decide the latter, we will employ
Parikh's Theorem.  

For each $w\in X^*$, let $\Psi(w)=(e_1,\ldots,e_k)$, where $e_i$ is the number
of occurrences of $a_i$ in $w$ for $1\le i\le k$. The resulting map $\Psi\colon
X^*\to \N^k$ is called the \emph{Parikh map}. Parikh's
Theorem~\cite{Parikh1966} states that for each context-free $L\subseteq X^*$,
its \emph{Parikh image} $\Psi(L)=\{\Psi(w) \mid w\in L\}$ is \emph{semilinear},
meaning that it is a finite union of sets of the form
\[ \{ v_0 ~+~ x_1\cdot v_1 ~+~ \cdots ~+~ x_n\cdot v_n \mid x_1,\ldots,x_n\in\N \}, \]
where $v_0\in\N^k$ and $v_1,\ldots,v_n\in \N^k$ are called the \emph{base
vectors} and the \emph{period vectors}, respectively. Again, Parikh's theorem
is effective, meaning that given a context-free grammar, we can compute base
vectors and period vectors for its semilinear Parikh image.

Furthermore, given a semilinear set $S\subseteq \N^k$, its complement
$\N^k\setminus S$ is effectively semilinear as well~\cite{GinsburgSpanier1966}.
Since $M=\{a_1\}^*\cdots\{a_k\}^*$ if and only if $\Psi(M)=\N^k$, we can
compute $\N^k\setminus \Psi(M)$ and check if it is non-empty. This concludes the proof of the theorem.
\end{proof}


\def\cprime{$'$} \def\cprime{$'$}

\end{document}